\documentclass[10pt]{amsart}

\usepackage[margin=1in]{geometry}
\usepackage{setspace}
\usepackage{color}
\usepackage{cancel}
\usepackage{hyperref}
\usepackage{graphicx}
\usepackage{diagbox}

\newtheorem{theorem}{Theorem}[section]
\newtheorem{lemma}[theorem]{Lemma}
\newtheorem{proposition}[theorem]{Proposition}
\newtheorem{corollary}[theorem]{Corollary}

\theoremstyle{definition}

\newtheorem{example}[theorem]{Example}

\theoremstyle{remark}
\newtheorem{remark}[theorem]{Remark}

\newcommand{\HH}{\mathrm{H}} \newcommand{\TT}{\mathrm{T}}

\DeclareMathOperator{\overlap}{ov}
\DeclareMathOperator{\Var}{Var}

\begin{document}

\title{Moments for generalizations of a coin flip game}
\author{Jia Huang}
\address{Department of Mathematics and Statistics, University of Nebraska at Kearney, Kearney, NE 68849, USA}
\email{huangj2@unk.edu}



\begin{abstract}
We derive a recursive formula for the moments of the number of flips using a possibly biased coin to produce a prescribed finite binary string $S$ when $S$ is either a run of heads or a run of heads followed by a tails.
Our recursive formula involve certain sums, which we simplify by using a one-parameter extension of the well-studied Eulerian number, which belongs to the two-parameter family of numbers introduced by Graham, Knuth, and Patashnik.
We also use the Goulden--Jackson cluster method and Fa\`a di Bruno's formula to establish a closed formula for the moments in a more general situation where a die having an arbitrary number of faces with possibly different probabilities is rolled repeatedly until a prescribed finite word occurs.
\end{abstract}

\maketitle

\section{Introduction}\label{sec:intro}

An X-post by Litt~\cite{Litt} in 2024 asked whether HH (two heads in a row) or HT (a heads followed by a tails) will occur more often when one flips a coin 100 times.
It attracted a lot of attentions, including some follow-up work from various perspectives by Basdevant--H\'enard--Maurel-S\'egala--Singh~\cite{BHMS}, Cheplyaka~\cite{Cheplyaka}, Ekhad--Zeilberger~\cite{EkhadZeilberger}, Grimmett~\cite{Grimmett}, Janson--Nica--Segert~\cite{JNS}, and Segert~\cite{Segert}, all in the same year.
A somewhat similar game concerning how many flips it will take to see the first occurrence of HH (or HT instead) was mentioned back in 2022 by Simonson~\cite[Ch.~2]{Simonson}.
In recent work~\cite{Huang} we derived some summation identities involving higher order Fibonacci numbers and certain variations of them from a more general coin flip game in which one keeps flipping a coin until a prescribed string $S$ of heads and tails occurs. 

Our work~\cite{Huang} on the coin flip game partially overlaps with work of Benjamin, Neer, Otero, and Sellers~\cite{ProbFibSum} on some weighted Fibonacci sums obtained from higher moments of the expected coin flips to produce $S=\HH\HH$.
Motivated by this connection, we now study a common generalization, i.e., the expected value $E(Y^n)$ of the $n$th power of the random variable $Y$ for the number of coin flips to obtain a prescribed string $S$ of finitely many heads and tails.
We also allow the coin to be biased, so the probability for heads and tails may be unequal.
Note that even using a fair coin, the starting side is slightly favored (with a probability of about $.51$) according to Diaconis, Holmes, and Montgomery~\cite{bias}.

We determine the moment $E(Y^n)$ recursively when $S=\HH^k$ (a run of $k$ heads) or $S=\HH^k\TT$ (a run of $k$ heads followed by a tails).
For small values of $k$ our result for $E(Y^n)$ coincides with some integer sequences on OEIS~\cite{OEIS}.
When $k$ is large we can simplify certain finite sums involved in our recursive formula by using an extension (depending on $k$) of the well-known \emph{Eulerian number}, which is interesting by itself as it satisfies a closed formula and some other identities that reduce to the alternating sum formula, the Carlitz identity, and the Worpitzky identity for the classical Eulerian number.
This can also be viewed as special case of a sequence defined by a two-parameter recurrence, first introduced by Graham, Knuth, and Patashnik~\cite{GKP} and later studied by others from various perspectives, such as Mansour--Shattuck~\cite{MansourShattuck}, Neuwirth~\cite{Neuwirth}, Spivey~\cite{Spivey}, and Wilf~\cite{Wilf}.

Nonetheless, it seems tedious to derive a closed formula for the moment $E(Y^n)$ from our recursive formula when $n$ gets big. 
It is also natural ask what happens when $S$ is an arbitrary finite string of heads and tails.
Furthermore, instead of a coin with only two sides, one can use a die with $m$ faces marked $1,2,\ldots,m$ and allow the probabilities of the faces to be distinct.
It turns out that a very different approach becomes effective in this much more general situation:
the Goulden--Jackson cluster method~\cite{GJ1, GJ2} (see also Guibas--Odlyzko~\cite{GuibasOdlyzko} and Noonan--Zeilberger~\cite{GJ3}) gives the generating function for words avoiding $S$ as a \emph{factor} (i.e., consecutive subword), and we connect these words (non-bijectively) to words with $S$ occurring exactly once at the end to obtain a closed formula for the moment $E(Y^n)$:
\begin{equation}
E(Y^n) = \sum_{\pi\in\Pi_n} |\pi|! \prod_{B\in \pi} \sum_{R\in\overlap(S)} \frac{(1-|R|)^{|B|} - (-|R|)^{|B|}}{P(R)}
\end{equation}
Here $\Pi_n$ consists of all partitions of the set $[n]:=\{1,2,\ldots,n\}$, $S$ is an arbitrary finite word on the alphabet $[m]$, $\overlap(S)$ consists of the \emph{overlaps} of $S$, i.e., the prefixes of $S$ that happen to be suffixes of $S$ as well, $|R|$ is the length of $R$, and $P(R)$ is the probability of $R$ when the die is rolled $|R|$ times.
The set partitions come into play due to Fa\`a di Bruno's formula, a generalization of the chain rule to higher order derivatives (for applications of Fa\`a di Bruno's formula to integer partitions, see recent work of Matsusaka~\cite{Matsusaka}).
In particular, the expected number $E(S)$ of turns to obtain $S$ is 
\begin{equation}
E(S) = E(Y) = \sum_{R\in\overlap(S)} \frac{1}{P(R)}.
\end{equation}

By our closed formula, $E(Y^n)$ remains the same when $S$ is reversed, and when the die is unbiased, $E(S)$ is asymptotically $m^{|S|}$ (plus some lower powers of $m$). 
This confirms affirmatively some conjectures on $E(S)$ for the unbiased coin flip game from earlier work~\cite{Huang}.
Another interesting consequence is that $E(S) = 1/P(S)$ when $\overlap(S)=\{S\}$ but not quite otherwise.

Furthermore, our closed formula for $E(S)$ provides a natural explanation for some previous results by Janson, Nica, and Segert~\cite{JNS} on which of two chosen words will occur more frequently when rolling the die many times.
It also gives an alternative way to solve a known problem: It will take on average $26^{11}+26^4+26$ times for a monkey to produce the word ABRACADABRA if the monkey randomly types a capital letter each time.
This answer follows immediately from our formula for $E(S)$ if one uses an unbiased die with $m=26$ faces and observes that $\overlap(\text{ABRACADABRA}) = \{\text{A, ABRA, ABRACADABRA}\}$.
The ABRACADABRA problem has been studied using probability with martingales~\cite{Lutsko, Williams}, and there might be existing work on the more general higher moment $E(Y^n)$ via a probabilistic approach, although we cannot find any.

This paper is structured as follows. 
In Section~\ref{sec:sum} we define a one-parameter extension of the Eulerian number and establish some summation identities related to it.
In section~\ref{sec:coin} we study the moment $E(Y^n)$ when a (possibly biased) coin is tossed repeatedly until $S$ occurs and use the results from Section~\ref{sec:sum} to help us determine $E(Y^n)$ recursively when $S=\HH^k$ or $S=\HH^k\TT$.
In Section~\ref{sec:die} we use the Goulden--Jackson cluster method and Fa\`a di Bruno's formula to obtain a closed formula for the moments $E(Y^n)$ in the more general situation when a die with $m$ faces is rolled repeatedly until a prescribed finite word $S$ on the alphabet $[m]$ occurs.
Finally, in Section~\ref{sec:questions} we give some questions for future study.

\section{An extension of the Eulerian number and some summation identities}\label{sec:sum}

In this section we establish some summation identities that will be used in Section~\ref{sec:coin}.
These identities involve a one-parameter extension of the well-known Eulerian number, which belongs to a two-parameter generalization introduced by Graham, Knuth, and Patashnik~\cite{GKP} and studied by many others~\cite{MansourShattuck, Neuwirth, Spivey, Wilf}.

Recall that the \emph{Eulerian number} $e_{n,i}$~\cite[A008292]{OEIS} is defined by $e_{0,0} := 1$, $e_{n,i} := 0$ whenever $i\le 0$ (unless $n=i=0$) or $i>n$, and $e_{n,i} := i e_{n-1,i} + (n-i+1) e_{n-1,i-1}$ for other values of $n$ and $i$;
see Table~\ref{tab:e}.

\begin{table}[h]
\footnotesize
\begin{tabular}{|l|c|c|c|c|c|c|c|c|c|}
\hline
\diagbox{$n$}{$i$} & $0$ & $1$ & $2$ & $3$ & $4$ & $5$ & $6$ & $7$ \\
\hline
$0$ & $1$ & & & & & & & \\
\hline
$1$ & & $1$ & & & & & & \\
\hline
$2$ & & $1$ & $1$ & & & & & \\
\hline 	
$3$ & & $1$ & $4$ & $1$ & & & & \\
\hline
$4$ & & $1$ & $11$ & $11$ & $1$ & & & \\
\hline
$5$ & & $1$ & $26$ & $66$ & $26$ & $1$ & & \\
\hline
$6$ & & $1$ & $57$ & $302$ & $302$ & $57$ & $1$ & \\
\hline
$7$ & & $1$ & $120$ & $1191$ & $2416$ & $1191$ & $120$ & $1$ \\
\hline
\end{tabular}
\vskip5pt
\caption{$e_{n,i}$ for small values of $n,i$}\label{tab:e}
\end{table}

We define an extension of $e_{n,i}$ depending on a parameter $k$ by $e_{0,0}^k := 1$, $e_{n,i}^k := 0$ whenever $i<0$ or $i>n$ and 
\[ 
e_{n,i}^k := (k+i+1)e_{n-1,i}^k + (n-k-i)e_{n-1,i-1}^k
\]
for other integer values of $n$ and $i$;
see Table~\ref{tab:ek} for some examples.
Note that $e_{n,i}^0=e_{n,i+1}$ for all $(n,i)\ne(0,0)$.

\begin{table}[h]
\footnotesize
\begin{tabular}{|l|c|c|c|c|c|c|}
\hline
\diagbox{$n$}{$i$} & $0$ & $1$ & $2$ & $3$ & $4$  \\
\hline
$0$ & $1$ & & & & \\
\hline
$1$ & $1+k$ & $-k$ & & & \\
\hline
$2$ & $(1+k)^2$ & $1-2k-2k^2$ & $k^2$ & & \\
\hline 	
$3$ & $(1+k)^3$ & $4-6k^2-3k^3$ & $1-3k+3k^2+3k^3$ & $-k^3$ & \\
\hline
$4$ & $(1+k)^4$ & $11+12k-6k^2-12k^3-4k^4$ & $11-12k-6k^2+12k^3+6k^4$ & $1-4k+6k^2-4k^3-4k^4$ & $k^4$ \\
\hline
\end{tabular}
\vskip5pt
\caption{$e_{n,i}^k$ for small values of $n,i$}\label{tab:ek}
\end{table}

We use $e_{n,i}$ and $e_{n,i}^k$ to study a finite sum, which will be used in Section~\ref{sec:coin}.

\begin{proposition}\label{prop:Carlitz}
For any integers $n,k\ge0$, we have
\[ 
\sum_{d=0}^k d^n x^d = \sum_{i=0}^n \frac{e_{n,i} x^i - e_{n,i}^k x^{k+i+1}}{(1-x)^{n+1}}.
\]
\end{proposition}

\begin{proof}
Fix an arbitrary integer $k\ge0$.
We prove the desired identity by induction on $n$.
It is trivial when $n=0$. 
Assume the result holds for an arbitrary $n\ge0$.
Applying the operator $x\frac{d}{dx}$ gives
\begin{align*}
\sum_{d=0}^k d^{n+1} x^d =& \sum_{i=0}^n \frac{ (1-x) 
\left( ie_{n,i} x^i - (k+i+1) e_{n,i}^k x^{k+i+1} \right) 
+ (n+1) \left( e_{n,i} x^{i+1} - e_{n,i}^k x^{k+i+2} \right) }
{(1-x)^{n+2}} \\
=& \sum_{i=0}^{n+1} \frac{ (ie_{n,i} + (n+2-i)e_{n,i-1})x^i 
- ( (k+i+1)e_{n,i}^k + (n+1-k-i) e_{n,i-1}^k ) x^{k+i+1} } 
{(1-x)^{n+2}} \\
=& \sum_{i=0}^{n+1} \frac{ e_{n+1,i} x^i 
-  e_{n+1,i}^k x^{k+i+1} } {(1-x)^{n+2}}. \qedhere
\end{align*}
\end{proof}

One can show by induction on $n$ that $e_{n,i}^k$ is a polynomial in $k$ of degree $n$ with special cases $e_{n,0}^k = (1+k)^n$ and $e_{i,i}^k = (-1)^i k^i$ for all integers $n,i\ge0$.
By taking a limit as $k\to\infty$ (for $|x|<1$) we derive from Proposition~\ref{prop:Carlitz} the well-known Carlitz identity for Eulerian numbers:
\[ \sum_{d=0}^\infty d^n x^d = \sum_{i=0}^n \frac{e_{n,i} x^i}{(1-x)^{n+1}}. \]
This infinite series was also studied by Sellers~\cite{Sellers}; for the case $x=1/2$ see Fulghesu--Sellers--Taylor~\cite{FulghesuSellersTaylor}.
Combining the Carlitz identity with Proposition~\ref{prop:Carlitz} immediately gives the following result.

\begin{corollary}\label{cor:sum}
For any integers $n,k\ge0$, we have
\[ \sum_{d=k+1}^\infty d^n x^d = \sum_{i=0}^n \frac{e_{n,i}^k x^{k+i+1}}{(1-x)^{n+1}}.
\]
\end{corollary}

Multiplying both sides of the identity in Corollary~\ref{cor:sum} by $(1-x)^{n+1}$ and extracting the coefficient of $x^{k+i+1}$ we obtain a closed formula for $e_{n,i}^k$.

\begin{corollary}\label{cor:ek}
For any integers $n,k,i\ge0$ we have
\[
e_{n,i}^k = \sum_{j=0}^{i} (-1)^j \binom{n+1}{j} (k+i+1-j)^n.
\]
\end{corollary}

Setting $k=0$ in Corollary~\ref{cor:sum} gives the well-known alternating sum formula for the Eulerian number:
\[
e_{n,i+1} = \sum_{j=0}^{i} (-1)^j \binom{n+1}{j} (i+1-j)^n.
\]
Setting $k=1$ in Corollary~\ref{cor:sum} gives another known sequence~\cite[A180246]{OEIS}.

We can also extract the coefficient of $x^{j+k+1}$ from Corollary~\ref{cor:sum} and obtain a different identity, which becomes the Worpitzky identity when $k=0$.

\begin{corollary}\label{cor:Worpitzky}
For any integers $n,k,j\ge0$ we have
\[
\sum_{i=0}^n e_{n,i}^k \binom{n+j-i}{n} = (j+k+1)^n .
\]
\end{corollary}

\begin{remark}
The \emph{ordered Bell numbers} or \emph{Fubini numbers}~\cite[A000670]{OEIS}
\begin{equation}\label{eq:Fubini}
b_n := \sum_{i=0}^\infty \frac{i^n}{2^{i+1}} =\sum_{i=0}^n e_{n,i}2^{n-i} = (1,1,3,13,75,541,4683,47293, \ldots)
\end{equation}
can be obtained from the Carlitz identity by setting $x=\frac12$
and is a special case of 
\[ c_{n,\ell} := \binom{n}{\ell} \sum_{i=0}^n \sum_{j=0}^{i} (-1)^j \binom{n+1}{j} (i+1-j)^\ell 2^{n-i} \]
since $c_{n,n}=b_n$.
Using this we can substitute $x=\frac12$ in Proposition~\ref{prop:Carlitz} and obtain
\begin{align*}
\sum_{d=0}^k \frac{d^n}{2^d} 
&= 2b_n -  \sum_{i=0}^n \sum_{j=0}^{i} (-1)^j \binom{n+1}{j} (k+i+1-j)^n 2^{n-k-i}  
= 2b_n - \sum_{\ell=0}^n \frac{c_{n,\ell} k^{n-\ell}}{2^k}.
\end{align*}
For $n>\ell\ge0$, computations suggest that $c_{n,\ell}$ agrees a known sequence~\cite[A202687]{OEIS} and has an interesting signed version~\cite[A162312]{OEIS}.
For example, we have
\[ \sum_{i=0}^k \frac{1}{2^i} = 2 - \frac{1}{2^k}, \quad
\sum_{i=0}^k \frac{i}{2^i} = 2 - \frac{k+2}{2^k}, \quad
\sum_{i=0}^k \frac{i^2}{2^i} = 6 - \frac{k^2+4k+6}{2^k}, \]
\[ \sum_{i=0}^k \frac{i^3}{2^i} = 26 - \frac{k^3+6k^2+18k+26}{2^k}, \quad
\sum_{i=0}^k \frac{i^4}{2^i} = 150 - \frac{k^4+8k^3+36k^2+104k+150}{2^k}, \]
\[ \sum_{i=0}^k \frac{i^5}{2^i} = 1082 - \frac{k^5+10k^4+60k^3+260k^2+750k+1082}{2^k}, \quad \ldots \]
\end{remark}



\section{Moments for a coin flip game}\label{sec:coin}

In this section we study the coin flip game in which one keeps flipping a coin with a probability $p$ for heads and $1-p$ for tails until a prescribed string $S$ consisting of finitely many heads and tails occurs.
Let $Y=Y(S)$ be the random variable for the number of coin flips to end this game throughout out this section.
We focus on the moments of $Y$, that is, the expected value $E(Y^n)$ of $Y^n$ for all positive integers $n$.
Let $Z(S)$ be the set of all finite strings of heads and tails with $S$ occurring exactly once at the end.
Then 
\[ E(Y^n) = \sum_{R\in Z(S)} |R|^n P(R) \]
where $|R|$ is the length of $R$ and $P(R)$ is the probability of $R$ when the coin is flipped $|R|$ times.
In particular, 
\begin{equation}\label{eq:E}
E(Y^n) = \sum_{d\ge0} \frac{d^n Z_d(S)}{2^d} \quad \text{for } p=\frac12
\end{equation}
where $Z_d(S)$ is the number of strings in $\{\HH,\TT\}^d$ with $S$ occurring exactly once at the end.
Benjamin, Neer, Otero, and Sellers~\cite{ProbFibSum} obtained certain weighted Fibonacci sums from Eq.~\eqref{eq:E} when $S=\HH\HH$.
In earlier work~\cite{Huang} we studied Eq.~\eqref{eq:E} when $n=1$ and $S$ either consists of at most four maximal runs of heads and tails or alternates between heads and tails.
Now we deal with the more general situation when $n$ is an arbitrary positive integer and $p$ is an arbitrary real number between $0$ and $1$.
We will derive a recursive formula for $E(Y^n)$ when $S$ is a run of heads or a run of heads followed by a tails.

\subsection{Game ended by a run of heads}
We first study the case when $Y=Y(S)$ for $S=\HH^k$.

\begin{theorem}\label{thm:Hk}
Suppose $S=\HH^k$ and $Y=Y(S)$. 
Then $E(Y^0)=1$ and for $n\ge1$, 
\begin{align*}
E(Y^n) &= k^n + \sum_{j=0}^{n-1} \binom{n}{j} E(Y^j) \sum_{i=1}^k \frac{(1-p)i^{n-j}}{p^{k-i+1}} \\
&= k^n + \sum_{j=0}^{n-1} \binom{n}{j} E(Y^j) \sum_{i=0}^n \frac{e_{n-j,i} p^i - e_{n-j,i}^k p^{k+i+1}}{(1-p)^{n-j}p^{k+1}}.
\end{align*}
\end{theorem}

\begin{proof}
It is clear that $E(Y^0)=1$. Assume $n\ge1$.
If the first $i$ flips give $\HH^{i-1}\TT$ for some $i\in\{1,2,\ldots,k\}$, then they are all wasted.
Otherwise the first $k$ flips give $\HH^k$, and the game ends right after.
Thus
\begin{align*}
E(Y^n) &= \sum_{i=1}^k p^{i-1}(1-p) E((Y+i)^n) + p^k k^n \\
&= \sum_{i=1}^k p^{i-1}(1-p) \sum_{j=0}^n \binom{n}{j} E(Y^j) i^{n-j} + p^k k^n \\
&=  \sum_{j=0}^{n-1} \sum_{i=1}^k p^{i-1}(1-p) \binom{n}{j} E(Y^j) i^{n-j} 
+ \sum_{i=1}^k p^{i-1}(1-p) E(Y^n) + p^k k^n  \\
&=  \sum_{j=0}^{n-1} \binom{n}{j} E(Y^j) \sum_{i=1}^k p^{i-1}(1-p) i^{n-j} 
+ (1-p^k) E(Y^n) + p^k k^n.
\end{align*}
This leads to the first desired recursive formula for $E(Y^n)$, and applying Proposition~\ref{prop:Carlitz} gives the second formula for $E(Y^n)$.
\end{proof}

\begin{corollary}\label{cor:Hk}
Suppose $S=\HH^k$ and $Y=Y(S)$. 
Then
\[ 
E(Y) = p^{-1} + p^{-2} + \cdots + p^{-k} = \frac{1-p^k}{(1-p)p^k},
\]
\[ 
E(Y^2) 
= \frac{2(1-p^k)}{(1-p)^2 p^{2k}} - \frac{2k+1-p^k}{(1-p)p^k},
\]
\begin{align*}
E(Y^3) 
&= \frac{6(1-p^k)}{(1-p)^3 p^{3k}} - \frac{12k+6-(6k+6)p^k}{(1-p)^2 p^{2k}} 
+ \frac{3k^2+3k+1-p^k}{(1-p)p^k},
\end{align*}
\begin{align*}
E(Y^4) =&  \frac{24(1-p^k)}{(1-p)^4 p^{4k}} - \frac{72k+36-(48k+36)p^k}{(1-p)^3 p^{3k}} \\
& + \frac{48k^2+48k+14-(12k^2+24k+14)p^k}{(1-p)^2 p^{2k}} 
- \frac{4k^3+6k^2+4k+1-p^k}{(1-p) p^k}.
\end{align*}
\end{corollary}

\begin{proof}
For $n=1$ we have the following by Theorem~\ref{thm:Hk}:
\[
E(Y) 
= k + \frac{1-p}{p^{k+1}} \sum_{i=1}^k ip^i 
= k + \frac{p +k p^{k+2} - (k+1) p^{k+1}}{(1-p)p^{k+1}}.
\]
Simplifying the above gives the desired formula for $E(Y)$.
It is similar (but more and more tedious) to obtain the desired formula for the higher moments.
\end{proof}

\begin{remark}
(i) Suppose $S=\HH^k$ and $Y=Y(S)$. 
If $p=1/2$ then Corollary~\ref{cor:Hk} gives 
\begin{itemize}
\item $E(Y) = 2^{k+1}-2 = 2 + 2^2 + \cdots + 2^k$,
\item $E(Y^2) = 2^{2k+3} - (2k+5) 2^{k+1} + 2$, 
\item $E(Y^3) = 3\cdot2^{3k+4} - 3(2k+3)2^{2k+3} + (3k^2+15k+13)2^{k+1} - 2$,
\item $E(Y^4) = 3\cdot2^{4k+7} - 3(6k+7)2^{3k+5} + (24k^2+72k+43)2^{2k+3} - (4k^3+30k^2+52k+29)2^{k+1} + 2$. 
\end{itemize}
On the other hand, by our recent work~\cite[Proposition~3.2]{Huang} we have
\begin{equation}\label{eq:H1}
E(Y^n) = \sum_{i=1}^\infty \frac{i^n F_{i-1}^k}{2^i} \quad\text{for } p = \frac12.
\end{equation} 
Here we use the \emph{the Fibonacci number of order $k$} defined by $F_i^k = 0$ for $i<k-1$, $F_i^k=1$ for $i=k-1$, and $F_i^k = F_{i-1}^k + F_{i-2}^k + \cdots + F_{i-k}^k$ for $i\ge k$; taking $k=2$ recovers the well-known \emph{Fibonacci number}.
Combining Theorem~\ref{thm:die} and Eq.~\eqref{eq:H1} we obtain the following when $p=1/2$ and $k$ is small.
\begin{itemize}
\item
If $p=\frac12$ and $k=1$ then $E(Y^n) = \sum_{i\ge1} \frac{i^n}{2^i} = (1, 2, 6, 26, 150, 1082, 9366, 94586, \ldots)$~\cite[A000629]{OEIS}.
\item
If $p=\frac12$ and $k=2$ then $E(Y^n) = \sum_{i\ge1} \frac{i^nF_{i-1}}{2^i} = (1,6, 58, 822, 15514, 366006, \ldots)$~\cite[A302922]{OEIS}; see also Benjamin, Neer, Otero, and Sellers~\cite{ProbFibSum}.
\end{itemize}

(ii) Suppose $S=\HH^k$, $Y=Y(S)$, and $p=1/m$ for some positive integer $m$.
Then $E(Y^n)$ is an integer for every $n\ge1$ by Theorem~\ref{thm:Hk}.
When $k$ is small we have the following special cases.
\begin{itemize}
\item
For $p=\frac13$ and $k=1$, $E(Y^n) = (1, 3, 15, 111, 1095, 13503, 199815, 3449631, \ldots)$~\cite[A201339]{OEIS}.
\item
For $p=\frac13$ and $k=2$, $E(Y^n) = (1, 12, 258, 8274, 353742, 18904602, 1212354798, 90706565514, \ldots)$~\cite[A352971]{OEIS}.
\item
For $p=\frac14$ and $k=1$, $E(Y^n) = (1, 4, 28, 292, 4060, 70564, 1471708, 35810212, \ldots)$~\cite[A201354]{OEIS}.
\item
For $p=\frac15$ and $k=1$, $E(Y^n) = (1, 5, 45, 605, 10845, 243005, 6534045, 204972605,\ldots)$~\cite[A201365]{OEIS}.
\end{itemize}
\end{remark}

\subsection{Game ended by a run of heads followed by a tails}
Now let $Y=Y(S)$ for $S=\HH^k\TT$. 

\begin{theorem}\label{thm:HkT}
Suppose $S=\HH^k\TT$ and $Y=Y(S)$. 
Then $E(Y^0)=1$ and for $n\ge1$, 
\begin{align*}
E(Y^n) &= \sum_{j=1}^{n-1} \binom{n}{j} E(Y^j) \sum_{i=1}^k \frac{(1-p) i^{n-j}}{p^{k-i+1}}
+ \sum_{i=1}^\infty \frac{(1-p) i^n}{p^{k-i+1}} \\
&= \sum_{j=1}^{n-1} \binom{n}{j} E(Y^j) \sum_{i=0}^n \frac{e_{n-j,i} p^i - e_{n-j,i}^k p^{k+i+1}}{(1-p)^{n-j}p^{k+1}}
+ \sum_{i=0}^n \frac{e_{n,i}p^i}{(1-p)^n p^{k+1}}.
\end{align*}
\end{theorem}

\begin{proof}
It is clear that $E(Y^0)=1$. Assume $n\ge1$ below.
There exists an integer $i\ge1$ such that the $i$th flip gives the first tails, or in other words, we have $\HH^{i-1}\TT$ from the first $i$ flips.
If $i\le k$ then the first $i$ flips are all wasted.
If $i>k$ then the game ends right after the first $i$ flips.
Thus 
\begin{align*}
E(Y^n) &= \sum_{i=1}^k p^{i-1}(1-p) E((Y+i)^n) + \sum_{i=k+1}^\infty p^{i-1}(1-p) i^n \\
& = \sum_{i=1}^k p^{i-1}(1-p) \sum_{j=0}^n \binom{n}{j} E(Y^j)i^{n-j} + \sum_{i=k+1}^\infty p^{i-1}(1-p) i^n \\
& = \sum_{i=1}^k p^{i-1}(1-p) \sum_{j=1}^n \binom{n}{j} E(Y^j)i^{n-j} + \sum_{i=1}^\infty p^{i-1}(1-p) i^n  \\
&= \sum_{j=1}^{n-1} \binom{n}{j} E(Y^j) \sum_{i=1}^k p^{i-1}(1-p) i^{n-j}
+ (1-p^k) E(Y^n) + \sum_{i=1}^\infty p^{i-1}(1-p) i^n.
\end{align*}
Solving for $E(Y^n)$ from the above gives the first desired formula for $E(Y^n)$, and applying Corollary~\ref{cor:sum} gives the second one.
\end{proof}

It is routine to obtain the following corollary from Theorem~\ref{thm:HkT}.

\begin{corollary}\label{cor:HkT}
Suppose $S=\HH^k\TT$ and $Y=Y(S)$. 
Then
\[ E(Y) = \frac{1}{(1-p)p^k}, \]
\[ E(Y^2) = \frac{2}{(1-p)^2 p^{2k}} - \frac{2k+1}{(1-p) p^k}, \]
\[ E(Y^3) = \frac{6}{(1-p)^3 p^{3k}} - \frac{12k+6}{(1-p)^2 p^{2k}}
+ \frac{3k^2+3k+1}{(1-p) p^k}, \]
\[ E(Y^4) = \frac{24}{(1-p)^4 p^{4k}} - \frac{72k+36}{(1-p)^3 p^{3k}} 
+ \frac{48k^2+48k+14}{(1-p)^2 p^{2k}} - \frac{4k^3+6k^2+4k+1}{(1-p) p^k}, \]
\[ E(Y^5) = \frac{120}{(1-p)^5 p^{5k}} - \frac{240(2k+1)}{(1-p)^4 p^{4k}} 
+ \frac{30(18k^2+18k+5)}{(1-p)^3 p^{3k}} - \frac{10(16k^3+24k^2+14k+3)}{(1-p)^2 p^{2k}}
+ \frac{(k+1)^5-k^5}{(1-p) p^k}. \]
\end{corollary}

Another corollary of Theorem~\ref{thm:HkT} involves the ordered Bell number~\cite[A000670]{OEIS} defined in Eq.~\eqref{eq:Fubini}.

\begin{corollary}
Suppose $S=\HH^k\TT$, $Y=Y(S)$, and $p=\frac12$.
For $n\ge1$ we have 
\[
E(Y^n) = 2^{k+1} b_n + \sum_{j=1}^{n-1} \binom{n}{j} E(Y^j) \sum_{i=1}^k i^{n-j} 2^{k-i}.
\]
\end{corollary}

\begin{remark}
For $Y=Y(\HH\TT)$ and $p = \frac12$, we have $E(Y^n) = (1, 4, 20, 124, 932, 8284, 85220, 997084, \ldots)$~\cite[A162509]{OEIS}.
For $T=T(\HH^k\TT)$ and $p=\frac12$, by our recent work~\cite[Proposition~3.2]{Huang} we have 
\[
E(Y^n) = \sum_{i=1}^\infty \frac{i^n \overline F_{i-2}^k}{2^i}.
\]
Here we use a variation of the Fibonacci number of order $k$ defined by $\overline F_i^k = 0$ for $i<k-1$, $\overline F_i^k = 1$ for $i=k-1$, and $\overline F_i^k = \overline F_{i-1}^k + \overline F_{i-2}^k + \cdots + \overline F_{i-k}^k + 1$, or equivalently, $\overline F_i^k = 2\overline F_{i-1}^k -  \overline F_{i-k-1}^k$, for $i\ge k$.
It turns out that $\overline F_i^k = F_0^k + F_1^k + \cdots + F_i^k$ is a partial sum of Fibonacci numbers of order $k$.
Some special cases include $\overline F_i^1=i+1$, $\overline F_i^2 = F_{i+2}-1$~\cite[A000071]{OEIS}, $\overline F_i^3$~\cite[A008937]{OEIS}, $\overline F_i^4$~\cite[A107066]{OEIS}, and so on.
The above expression of $E(Y^n)$ in terms of $\overline F_i^k$ holds because $\overline F_i^k$ is the number of binary strings of length $i+2$ with $1^k 0$ ($k$ consecutive ones followed by a zero) occurring precisely at the end.
By removing $1^k0$ we have a bijection from such binary strings to binary strings of length $i-k+1$ avoiding $1^k0$ as a (consecutive) substring.
This has an interesting connection with a recent result of Bates, Morrison, Rogers, Serafini, and Sood~\cite{Bates} asserting that the number of equivalence classes on binary strings of length $i-k+1$ for the equivalence relation $\sim_a$ is also $\overline F_i^k$, where $\sim_a$ is defined by swapping $0$ and $1$ in any occurrence of a chosen binary string $a$ of length $k+1$.
This coincidence has a nice explanation when $a=1^k 0$ since each equivalence class of $\sim_a$ is uniquely represented by a binary string avoiding $a$ in this special case, but the same does not hold for an arbitrary keyword $a$.
\end{remark}

\section{Rolling a die}\label{sec:die}

Let $m$ be a positive integer and write $[m]:=\{1,2,\ldots,m\}$.
Suppose $S$ is a word on the alphabet $[m]$ with a finite length $|S|$.
In this section we study the game in which a player keeps throwing a die with $m$ faces marked $1,2,\ldots,m$ until $S$ occurs;
when $m=2$ this becomes the coin flip game studied earlier.
Suppose that each possible outcome $i\in[m]$ occurs with a probability $p_i$, where $p_1, \ldots, p_m$ are nonnegative real numbers satisfying $p_1+\cdots+p_m=1$.
In particular, we have $p_1=\cdots=p_m=1/m$ in the unbiased situation.

\subsection{The expected number of turns to end the game}
We first study the expected number of turns for the game mentioned above to end, that is,
\[ 
E(S) := \sum_{w\in Z(S)} |w|\cdot P(w).
\] 
Here $Z(S)$ is the set of all finite words on the alphabet $[m]$ with $S$ occurring exactly once at the end, and for each word $w\in Z(S)$, $|w|$ is the length of $w$ and $P(w)$ is the probability of $w$ occurring when rolling the die exactly $|w|$ times.

To establish a closed formula for $E(S)$, we need some more definitions.
Let $x_1, \ldots, x_m$ be commuting variables.
For any finite word $w=w_1\cdots w_\ell$ on $[m]$, define $x_w:=x_{w_1}\cdots x_{w_m}$.
Then evaluating $x_w$ at $x_i=p_i$ for all $i=1,\ldots, m$ gives $P(w)$.
A \emph{factor} of $S$ is a (possibly empty) consecutive subword of $S$.
For example, the word $13211$ has $121$ as a subword (but not a factor) and $321$ as a factor.
A \emph{prefix} (resp., \emph{suffix}) of $S$ is a factor with which $S$ begins (resp., ends).
If $R$ is a prefix (resp, suffix) of $S$, then let $R\setminus S$ (resp., $S/R$) be the factor of $S$ obtained by deleting the prefix (resp., suffix) $R$.
For instance, $13$ is a prefix of $13211$, $11$ is a suffix of $13211$, $13\setminus 13211 = 211$, and $13211/11=132$.
A nonempty factor of $S$ that is simultaneously a prefix and a suffix of $S$ is called an \emph{overlap} of $S$.
Let $\overlap(S)$ denote the set of all overlaps of $S$, e.g., $\overlap(13211) = \{1, 13211\}$.


\begin{theorem}\label{thm:die}
Let $S$ be a word of a finite length on the alphabet $[m]$.
Then $E(S) = \sum_{R\in\overlap(S)} P(R)^{-1}$.
\end{theorem}

\begin{proof}
Let $X(S)$ be the set of all finite words on $[m]$ avoiding $S$ as a factor.
Applying the Goulden--Jackson Cluster Method~\cite{GJ1, GJ2} (see also Noonan and Zeilberger~\cite{GJ3}) we obtain the generating function for $X(S)$:
\[ 
\sum_{w\in X(S)} x_w
= \left( 1 - \sum_{i=1}^m x_i + \frac{x_S} {\sum_{R\in\overlap(S)} x_S/x_R } \right)^{-1}.
\]

Given a word on $[m]$ avoiding $S$ as a factor, appending $S$ gives a new word in which the first occurrence of $S$ and the appended copy of $S$ must overlap at some $R\in\overlap(S)$, and deleting $R\setminus S$ at the end of this word produces a word with $S$ occurring exactly once at the end.
This map is not one-to-one, but every preimage of a word on $[m]$ with $S$ occurring exactly once at the end can be obtained by removing some $R\in\overlap(S)$ at the end of this word. 
Thus
\[ 
\sum_{w\in Z(S)} x_w \sum_{R\in\overlap(S)} \frac{1}{x_R} = \sum_{w\in X(S)} x_w
= \left( 1 - \sum_{i=1}^m x_i + \frac{1}{\sum_{R\in\overlap(S)} \frac1{x_R}} \right)^{-1}.
\]
This implies that
\[ 
\sum_{w\in Z(S)} x_w
= \left( 1 + \left(1 - \sum_{i=1}^m x_i\right) \sum_{R\in\overlap(S)} \frac1{x_R} \right)^{-1}.
\]
Replacing each $x_i$ with $x x_i$, where $x$ is a new variable commuting with $x_1,\ldots, x_m$, we obtain
\begin{equation}\label{eq:Z} 
\sum_{w\in Z(S)} x^{|w|} x_w
= \left( 1 + \left(1 - x\sum_{i=1}^m x_i\right) \sum_{R\in\overlap(S)} \frac{1}{x^{|R|} x_R} \right)^{-1}.
\end{equation}
We can then obtain $E(S)$ by differentiating the above generating function with respect to $x$ and substituting in $x=1$ and $x_i=p_i$ for $i=1,\ldots,m$.
Since $x$ is independent of $x_1, \ldots, x_m$, we may first do the substitution $x_i=p_i$ and simplify Eq.~\eqref{eq:Z} to the following function of a single variable $x$ before differentiation:
\begin{equation}\label{eq:h}
h(x) := u^{-1} = \left( 1 + \sum_{R\in\overlap(S)} \frac{1-x}{x^{|R|} P(R)} \right)^{-1}.
\end{equation}
To find $h'(1)$, note that
\[ u = 1 + \sum_{R\in\overlap(S)} \frac{1-x}{x^{|R|} P(R)} \ \Longrightarrow \ 
u' = \sum_{R\in\overlap(S)} \left( \frac{-|R|}{x^{|R|+1} P(R)} - \frac{1-|R|}{x^{|R|} P(R)} \right),
\]
Setting $x=1$ gives $u(1) = 1$ and $u'(1) = \sum_{R\in\overlap(S)}\frac{-1}{P(R)}$.
Thus evaluating $h'(x) = -u^{-2} u'$ at $x=1$ gives
\[ 
E(S) = \sum_{w\in Z(S)} |w| P(w) = \sum_{R\in\overlap(S)} \frac1{P(R)}. \qedhere
\]
\end{proof}

Theorem~\ref{thm:die} recovers the results on $E(S)$ in our earlier work~\cite{Huang} when $S$ consists of at most four runs of heads and tails or alternates between heads and tails.
It also immediately implies the following corollary, which confirms affirmatively a previous conjecture~\cite{Huang} (in the unbiased coin flip case): $E(S)$ is asymptotically $m^s$ (plus some lower powers of $m$) for any binary string $S$ of length $s$ if an unbiased die with $m$ faces is used.

\begin{corollary}
For the game of rolling a die with $m$ faces unbiasedly, we have $E(S) = \sum_{R\in\overlap(S)} m^{|R|}$.
\end{corollary}

\begin{remark}
By Theorem~\ref{thm:die}, if $\overlap(S) = \{S\}$ then $E(S) = 1/ P(S)$.
This idea was implicit in the 2025 MATHCounts State Competition Target Round Problem \#8: throwing an unfair coin three times with a $\frac23$ chance of getting a heads each time gives $\HH\TT\TT$ with a probability of $\frac23\cdot\frac13\cdot\frac13=\frac2{27}$, so the expected number of coin flips to obtain $\HH\TT\TT$ is $\frac{27}{2}$.
However, this solution (temporarily available on the MATHCounts official website) 
has no mentioning of the crucial condition $\overlap(S)=\{S\}$, without which the reciprocal relation would not hold by Theorem~\ref{thm:die}.
\end{remark}

Another consequence of Theorem~\ref{thm:die} is given below, which was conjectured and naively explained for the unbiased coin case in our earlier work~\cite{Huang}.

\begin{corollary}
If $S'$ is the reversal of a finite word on the alphabet $[m]$ then $E(S')=E(S)$.
\end{corollary}

Moreover, the extrema of $E(S)$ when $S$ has a fixed length also follow from Theorem~\ref{thm:die}.

\begin{corollary}\label{cor:extrema}
(i) Let $p:= \min\{p_1,\ldots,p_m\} = p_i$. 
Then the following is achieved when $S$ is a run of $i$'s:
\[
\max \left\{E(S): S\in[m]^k \right\} = 
p^{-1} + p^{-2} + \cdots + p^{-k} = \frac{1-p^k}{(1-p)p^k}.
\]
(ii) If $m>1$ and $p_1=\cdots=p_m$ then $\min \left\{E(S): S\in[m]^k \right\}= m^k$ is attained when $\overlap(S)=\{S\}$, e.g., $S$ consists of two maximal runs.
\end{corollary}

\begin{remark}
Corollary~\ref{cor:extrema} (ii) does not hold in the biased situation.
For example, if $m=3$, $p_1=\frac12$, $p_2=p_3=\frac14$, and $k\ge2$ then $E(1^{k-1}2) = E(1^{k-1}3) = 2^{k+1}$ whereas $E(1^k) = 2 + 2^2 +\cdots + 2^k = 2^{k+1}-2$.
\end{remark}

\subsection{Higher Moments}
Let $Y=Y(S)$ be the random variable for the number of turns to end the game described at the beginning of this section.
We establish a closed formula for the higher moment $E(Y^n)$ by extending the proof of Theorem~\ref{thm:die} and using a generalization of the chain rule for higher order derivatives known as Fa\`a di Bruno's formula:
\[
\frac{d^n}{dx^n} f(g(x)) = \sum_{\pi\in\Pi_n} f^{(|\pi|)}(g(x))
\prod_{B\in \pi} g^{(|B|)}(x).
\] 
Here $\Pi_n$ consists of all partitions of the set $[n]:=\{1,\ldots,n\}$, and each set partition $\pi\in\Pi_n$ consists of some (unordered) blocks $B_1,\ldots, B_{|\pi|}$.
We need to adapt Fa\`a di Bruno's formula to the operator $\Omega:=x\frac{d}{dx}$.

\begin{lemma}\label{lem:Omega}
Applying the operator $\Omega:=x\frac{d}{dx}$ to a composite function $f\circ g$  repeatedly gives the following for all $n\ge1$ (assuming the existence of all necessary derivatives of $f$ and $g$):
\[
\Omega^n (f\circ g) = \sum_{\pi\in\Pi_n} \left( f^{(|\pi|)}\circ g \right) \prod_{B\in \pi} \Omega^{|B|}g.
\] 
\end{lemma}

\begin{proof}
The desired formula is trivial when $n=1$. 
Assume it holds for $n=k\ge1$.
Then
\begin{align*}
\Omega^{k+1} (f\circ g) &= 
\sum_{\pi\in\Pi_k} \left( \left( f^{(|\pi|+1)}\circ g \right) \Omega g \prod_{B\in \pi} \Omega^{|B|}g
+ \left( f^{(|\pi|)}\circ g \right) \sum_{A \in \pi} \Omega^{|A|+1}g \prod_{B\in \pi\setminus\{A\}} \Omega^{|B|}g \right) \\
&= \sum_{\pi\in\Pi_{k+1}} \left( f^{(|\pi|)}\circ g \right) \prod_{B\in \pi} \Omega^{|B|}g.
\end{align*}
Here the second equality holds since partitions of $[k+1]$ can be obtained from partitions of $[k]$ by either creating a new block $\{k+1\}$ or inserting $k+1$ to an existing block.
\end{proof}

\begin{theorem}\label{thm:moment}
For every integer $n\ge1$ we have 
\[ 
E(Y^n) = \sum_{\pi\in\Pi_n} |\pi|! \prod_{B\in \pi} \sum_{R\in\overlap(S)} \frac{(1-|R|)^{|B|} - (-|R|)^{|B|}}{P(R)}
\]
\end{theorem}

\begin{proof}
By the proof of Theorem~\ref{thm:die}, applying the operator $\Omega$ to $h=u^{-1}$ defined in Eq.~\eqref{eq:h} $n$ times gives
\[ E(Y^n) = \sum_{w\in Z(S)} |w|^n P(w) = \left(\Omega^n h\right)|_{x=1}.
\]
By Lemma~\ref{lem:Omega},
\[ 
\Omega^n u^{-1} = 
\sum_{\pi\in\Pi_n} (-1)^{|\pi|} |\pi|! u^{-|\pi|-1}
\prod_{B\in \pi} \Omega^{|B|}u.
\]
Setting $x=1$ gives the desired formula for $E(Y^n)$ since
\[ u = 1 + \sum_{R\in\overlap(S)} \frac{1-x}{x^{|R|} P(R)} \ \Longrightarrow \ 
\Omega^i u = \sum_{R\in\overlap(S)} \left( \frac{(-|R|)^i}{x^{|R|} P(R)} - \frac{(1-|R|)^i}{x^{|R|-1} P(R)} \right).
\qedhere \]
\end{proof}

\begin{example}
We already have $E(Y) = E(S) = \sum_{R\in\overlap(S)} P(R)^{-1}$ by Theorem~\ref{thm:die}.
For $n=2$ we have $\Omega^2 h = 2u^{-3}(\Omega u)^2 - u^{-2}\Omega^2 u$ and
\[
E(Y^2) = 2\left(\sum_{R\in\overlap(S)} \frac{1}{P(R)} \right)^2
+ \sum_{R\in\overlap(S)} \frac{(1-|R|)^2-|R|^2}{P(R)}.
\]
For $n=3$ we have $\Omega^3 h = -6u^{-4}(\Omega u)^3 + 6u^{-3} \Omega u \Omega^2 u - u^{-2}\Omega^3 u$ and
\[
E(Y^3) = 6\left(\sum_{R\in\overlap(S)} \frac{1}{P(R)} \right)^3
+ 6 \sum_{R\in\overlap(S)} \frac{1}{P(R)} \sum_{R\in\overlap(S)} \frac{(1-|R|)^2-|R|^2}{P(R)}
+ \sum_{R\in\overlap(S)} \frac{(1-|R|)^3+|R|^3}{P(R)}.
\]
For $n=4$ we have $\Omega^4 h = 24u^{-5}(\Omega u)^4 - 36 u^{-4}(\Omega u)^2 \Omega^2 u + 6u^{-3} (\Omega^2 u)^2 + 8u^{-3} \Omega u \Omega^3 u - u^{-2}\Omega^4 u$ and
\begin{multline*}
E(Y^4) = 24\left(\sum_{R\in\overlap(S)} \frac{1}{P(R)} \right)^4
+36 \left(\sum_{R\in\overlap(S)} \frac{1}{P(R)} \right)^2 \sum_{R\in\overlap(S)} \frac{(1-|R|)^2-|R|^2}{P(R)} \\
+6 \left(\sum_{R\in\overlap(S)} \frac{(1-|R|)^2-|R|^2}{P(R)}\right)^2 
+8 \sum_{R\in\overlap(S)} \frac{1}{P(R)} \sum_{R\in\overlap(S)} \frac{(1-|R|)^3+|R|^3}{P(R)}
+ \sum_{R\in\overlap(S)} \frac{(1-|R|)^4 - |R|^4}{P(R)}.
\end{multline*}
\end{example}

The following corollaries follow immediately from Theorem~\ref{thm:moment}.
\begin{corollary}
The moment $E(Y^n)$ remains the same when $S$ is reversed.
\end{corollary}

\begin{corollary}
The variance of the random variable $Y=Y(S)$ is 
\[ 
\Var(Y) = E(Y^2) - E(Y)^2 
= \left(\sum_{R\in\overlap(S)} \frac{1}{P(R)} \right)^2
+ \sum_{R\in\overlap(S)} \frac{(1-|R|)^2-|R|^2}{P(R)}.
\]
\end{corollary}

For the coin flip game studied in Section~\ref{sec:coin} we have the following corollary of 
Theorem~\ref{thm:moment}, which recovers Corollary~\ref{cor:Hk} and Corollary~\ref{cor:HkT}.

\begin{corollary}
In the coin flip game ended by $S$ with probability $p$ for heads, the random variable $Y=Y(S)$ satisfies
\[
E(Y^n) = \sum_{\pi\in\Pi_n} (-1)^{n-|\pi|} |\pi|! \prod_{B\in \pi} \sum_{i=1}^k \frac{i^{|B|} - (i-1)^{|B|}}{p^i}
\quad \text{for } S=\HH^k,
\]
\[ 
E(Y^n) = \sum_{\pi\in\Pi_n} (-1)^{n-|\pi|} |\pi|! \prod_{B\in \pi} \frac{(k+\ell)^{|B|}-(k+\ell-1)^{|B|}}{p^k(1-p)^\ell}
\quad \text{for } S=\HH^k\TT^\ell.
\]
\end{corollary}

\section{Questions for future research}\label{sec:questions}

There is a nice combinatorial interpretation for the Eulerian number $e_{n,i}$: It enumerates permutations of $n$ whose number of descents is $i-1$, where the number of descents can be replaced by any other Eulerian statistics, such as the number of excedances.
It would be nice to find a similar combinatorial interpretation for our one-parameter extension $e_{n,i}^k$.
Note that $e_{n,i} = e_{n,n-i+1}$ since reversing a permutation (in one-line notation) with $i-1$ descents gives a permutation with $n-i$ dscents, but $e_{n,i}^k$ does not carry this kind of symmetry (see Table~\ref{tab:ek}).


One can generalize the die rolling game studied in this paper by allowing it to terminate upon the first appearance of any element of a finite set $\mathcal{S}$ of prescribed words.
Let $Y=Y(\mathcal{S})$ be the random variable for the number of turns to end this more general game.
It might be possible to determine $E(Y^n)$ by extending our proof of Theorem~\ref{thm:die} since the generating function for finite words in $[m]$ avoiding every element of $S$ as a factor can be derived from the Goulden--Jackson Cluster Method~\cite{GJ1, GJ2}, although the result may become complicated when different words in $\mathcal{S}$ overlap with each other~\cite[examples on p.~9]{GJ3}.


The game studied in this paper is similar to another game played by two people as mentioned in Section~\ref{sec:intro}: 
Two players respectively pick two words of the same lengths on the alphabet $[m]$ and take turns to roll a die with $m$ faces marked $1,2,\ldots,m$ for $n$ times, the player seeing their word appearing more frequently being the winner.
Ekhad--Zeilberger~\cite{EkhadZeilberger}, Grimmett~\cite{Grimmett}, and Segert~\cite{Segert} provided asymptotic formulae (as $n\to\infty$) for the probabilities of $\HH\HH$ outnumbering $\HH\TT$ or the other way around in the unbiased coin flip case.
Janson, Nica, and Segert~\cite{JNS} generalized the formulae to the unbiased die rolling situation with two arbitrary words of the same length picked by the two players, and the result depends only on the number of overlaps of each of the two words.
In particular, the winner's word must have a smaller number of overlaps.
This can be naturally explained by our work in this paper since Theorem~\ref{thm:die} implies that the average number of occurrences of a word $S$ after rolling the die $n$ times is $n/E(S) = n \big/ \sum_{R\in\overlap(S)} P(R)^{-1}$, which becomes larger when $\overlap(S)$ gets smaller.
It would be nice to use the results in this paper to recover the asymptotic formulae mentioned above and derive more consequences in that direction.

\section*{Acknowledgments}

The author is grateful for email exchanges with James Sellers and with Erik Bates.


\begin{thebibliography}{99}


\bibitem{BHMS}
A-L. Basdevant, O. H\'enard, \'E. Maurel-S\'egala, and A. Singh, On cases where Litt’s game is fair, C. R. Math. Acad. Sci. Paris {\bf 363} (2025), 977--984.

\bibitem{Bates}
E. Bates, B. Morrison, M. Rogers, A. Serafini, and A. Sood, A new combinatorial interpretation of partial sums of $m$-step Fibonacci numbers, J. Integer Seq. {\bf 28} (2025), no.~6, Art. 25.6.6, 22 pp.

\bibitem{ProbFibSum}
A. T. Benjamin, J. D. Neer, D. T. Otero, and J. A. Sellers, A probabilistic view of certain weighted Fibonacci sums, Fibonacci Quart. {\bf 41} (2003), no.~4, 360--364.

\bibitem{Cheplyaka}
R. Cheplyaka, Alice and Bob flipping coins puzzle, March 18, 2024,
\url{https://ro-che.info/articles/2024-03-18-alice-bob-coin-flipping}.

\bibitem{bias}
P.~W. Diaconis, S.~P. Holmes and R.~W. Montgomery, Dynamical bias in the coin toss, SIAM Rev. {\bf 49} (2007), no.~2, 211--235.
 
\bibitem{EkhadZeilberger}
S. B. Ekhad and D. Zeilberger, How to answer questions
of the type: if you toss a coin n times, how likely is HH to show up
more than HT? 
The personal Journal of Shalosh B. Ekhad and Doron Zeilberger, May 20, 2024,
\url{https://sites.math.rutgers.edu/~zeilberg/mamarim/mamarimhtml/litt.html}

\bibitem{FulghesuSellersTaylor}
D. Fulghesu, J. A. Sellers, and C. K. Taylor, Infinite families of infinite series with integer sums, College Math. J. {\bf 54} (2023), 33-–43.


\bibitem{GJ1}
I.~P. Goulden and D.~M. Jackson, An inversion theorem for cluster decompositions of sequences with distinguished subsequences, J. London Math. Soc. (2) {\bf 20} (1979), no.~3, 567--576.

\bibitem{GJ2}
I.~P. Goulden and D.~M. Jackson, {\it Combinatorial enumeration}, Wiley-Interscience Series in Discrete Mathematics, Wiley, New York, 1983.

\bibitem{GKP}
R.~L. Graham, D.~E. Knuth and O. Patashnik, {\it Concrete mathematics}, second edition, 
Addison-Wesley, Reading, MA, 1994

\bibitem{Grimmett}
G.~R. Grimmett, Alice and Bob on $\Bbb{X}$: Reversal, Coupling, Renewal, Amer. Math. Monthly {\bf 133} (2026), no.~5, 439--451.

\bibitem{GuibasOdlyzko}
L.~J. Guibas and A.~M. Odlyzko, String overlaps, pattern matching, and nontransitive games, J. Combin. Theory Ser. A {\bf 30} (1981), no.~2, 183--208

\bibitem{Huang}
J. Huang, A coin flip game and generalizations of Fibonacci numbers, arXiv:2501.07463, to appear in Fibonacci Quart.

\bibitem{JNS}
S. Janson, M. Nica, and S. Segert, The generalized Alice HH vs Bob HT problem, J. Theoret. Probab. {\bf 38} (2025), no.~4, Paper No. 83, 52 pp.

\bibitem{Litt}
D. Litt, $\Bbb{X}$-post, March 16 2024, \url{https://x.com/littmath/status/1769044719034647001}.

\bibitem{Lutsko}
C. Lutsko, Monkeys typing and martingales, \url{https://chrislutsko.com/files/Abracadabra.pdf}.



\bibitem{MansourShattuck}
T. Mansour and M.~A. Shattuck, A combinatorial approach to a general two-term recurrence, Discrete Appl. Math. {\bf 161} (2013), no.~13-14, 2084--2094.

\bibitem{Matsusaka}
T. Matsusaka, Applications of Fa\`a{} di Bruno's formula to partition traces, Res. Number Theory {\bf 11} (2025), no.~3, Paper No. 69, 11 pp.

\bibitem{Neuwirth}
E. Neuwirth, Recursively defined combinatorial functions: extending Galton's board, Discrete Math. {\bf 239} (2001), no.~1-3, 33--51.

\bibitem{GJ3}
J. Noonan and D. Zeilberger, The Goulden-Jackson cluster method: extensions, applications and implementations, J. Differ. Equations Appl. {\bf 5} (1999), no.~4-5, 355--377.

\bibitem{Segert}
S. Segert, A proof that HT is more likely to outnumber HH than vice versa in a string of n coin flips, arXiv:2405.16660 [math.CO].

\bibitem{Sellers}
J. A. Sellers, Beyond Mere Convergence,  PRIMUS (Problems, Resources, and Issues in Mathematics Undergraduate Studies), {\bf XII} (2002), 157–-164.

\bibitem{Simonson}
S. Simonson, {\it Looking for Math in All the Wrong Places: Math in Real Life}, AMS/MAA Spectrum Vol. 104, American Mathematical Society, 2022.

\bibitem{OEIS}
N. J. A. Sloane et al., 
{\it The On-Line Encyclopedia of Integer Sequences},
OEIS Foundation Inc., \url{https://oeis.org}.

\bibitem{Spivey}
M.~Z. Spivey, On solutions to a general combinatorial recurrence, J. Integer Seq. {\bf 14} (2011), no.~9, Article 11.9.7, 19 pp.


\bibitem{Wilf}
H.~S. Wilf, The method of characteristics, and ``problem 89'' of Graham, Knuth and Patashnik, arXiv:math/0406620. 


\bibitem{Williams}
D. Williams, {\it Probability with martingales}, Cambridge Mathematical Textbooks, Cambridge Univ. Press, Cambridge, 1991.

\end{thebibliography}
\end{document}